\documentclass[12pt]{article}

\usepackage[a4paper,total={6.3in,9.2in}]{geometry}

\usepackage{amsmath,amsthm,amssymb,amsfonts}
\usepackage{mathrsfs}
\usepackage{extarrows}
\usepackage{mathtools} 

\usepackage{tikz}
\usetikzlibrary{positioning}
\usetikzlibrary{arrows.meta,decorations.pathreplacing,calligraphy}

\usepackage{xcolor}
\usepackage{enumerate}
\usepackage{float}

\usepackage[numbers,square,comma,sort&compress]{natbib}
\usepackage[hidelinks]{hyperref}

\newtheorem{thm}{Theorem}[section]
\newtheorem{lem}[thm]{Lemma}
\newtheorem{cor}[thm]{Corollary}
\newtheorem{prop}[thm]{Proposition}

\newtheorem{prob}[thm]{Problem}

\theoremstyle{definition}
\newtheorem{defi}[thm]{Definition}
\newtheorem{exa}[thm]{Example}
\newtheorem{rem}[thm]{Remark}

\newcommand{\defn}[1]{\emph{#1}}

\newcommand{\tre}[1]{\textcolor{red}{#1}}

\newcommand{\ga}{\gamma}
\newcommand{\Ga}{\Gamma}
\newcommand{\lam}{\lambda}
\newcommand{\sig}{\sigma}

\newcommand{\T}{\mathbb{T}}

\newcommand{\cB}{\mathcal{B}}

\newcommand{\cD}{\mathcal{D}}

\newcommand{\cM}{\mathcal{M}}

\newcommand{\cS}{\mathcal{S}}

\newcommand{\cU}{\mathcal{U}}

\newcommand{\sA}{\mathscr{A}}
\newcommand{\sB}{\mathscr{B}}
\newcommand{\sC}{\mathscr{C}}

\newcommand{\UB}{{\rm UB}}

\newcommand{\floor}[1]{\left\lfloor #1 \right\rfloor}

\newcommand{\xto}{\xlongrightarrow}

\newcommand{\cover}{\lessdot}
\newcommand{\join}{\vee}
\newcommand{\meet}{\wedge}

\numberwithin{equation}{section}

\title{Combinatorial aspects of the Delannoy Lattice}
\author{Xi Chen, Yuxian Dong
\thanks{
	Corresponding author.
    \newline\hspace*{3mm}
    {\it Email address:}\quad
    chenxi@dlut.edu.cn (X. Chen),
    yuxiand@hotmail.com (Y. Dong)
    }
}
\date{\footnotesize
School of Mathematical Sciences, Dalian University of Technology, Dalian 116024, PR China}

\begin{document}
\maketitle

\begin{abstract}
The Delannoy numbers $d(n,k)$ count lattice paths from $(0,0)$ to $(n-k,k)$ using steps $(1,0),(0,1)$ and $(1,1)$.
This paper introduces a graded poset $\cD_n$ on the Delannoy paths ending on the line $x+y=n$,
whose rank-generating function is the Delannoy polynomial $d_n(x)=\sum_{k=0}^n d(n,k)x^k$.
We prove that $\cD_n$ is a self-dual lattice,
which is call the Delannoy lattice.
By establishing an explicit symmetric Boolean decomposition,
we show that $\cD_n$ is a symmetric Boolean order,
thereby recovering the $\ga$-positivity of $d_n(x)$.
Such a decomposition is refined to a symmetric chain decomposition with the chain cover property,
and is applied to determine all maximum antichains.
We also investigate other combinatorial aspects of $\cD_n$,
including supersolvability, the M\"obius number, characteristic polynomials, and zeta polynomials.
\\[1pt]
\end{abstract}

\section{Introduction}

A \defn{Delannoy path of size $n$} is a lattice path from $(0,0)$ to some point on the line $x+y=n$ using horizontal step $H=(1,0)$, vertical step $V=(0,1)$, and diagonal step $D=(1,1)$.
We identify a Delannoy path with its word over $\{H,V,D\}$.
Let $\cD_n$ be the set of all Delannoy paths of size $n$.
For $p\in\cD_n$, let $h(p)$, $v(p)$, and $d(p)$ denote the numbers of $H$-, $V$-, and $D$-steps in $p$, respectively.
Then
\[
h(p)+v(p)+2d(p)=n.
\]
Define an \defn{elementary move} in $\cD_n$ as one of the following three cases
\begin{equation}\label{eq-cover}
    aHb \to aVb,\qquad
    aHHb \to aDb,\qquad
    aDb \to aVVb,
\end{equation}
where $a$ and $b$ are arbitrary words over $\{H,V,D\}$.
Each elementary move preserves $h+v+2d=n$,
so it maps $\cD_n$ to itself.
We define a relation $\le$ on $\cD_n$ by $p\le q$ if and only if
there exists a path sequence
\[
p=p_0\to p_1\to\cdots\to p_{\ell}=q
\]
for some $\ell\ge0$.
Reflexivity and transitivity are immediate from definition.
If $p\le q$ and $q\le p$ then we have $v(p)+d(p)=v(q)+d(q)$
since every elementary move increases $v+d$ by exactly one.
Thus $p=q$, which proves antisymmetry.
Hence, $\le$ gives a partial order on $\cD_n$,
and therefore induces a poset.
In the next section, we prove that $\cD_n$ is a self-dual lattice,
which we call the \defn{Delannoy lattice}.

The Delannoy numbers $d(n,k)$ count the Delannoy paths from $(0,0)$ to $(n-k,k)$.
These numbers are named after Henri Delannoy
(see~\cite{BS05} for historical remarks).
The generating functions
\[
d_n(x)=\sum_{k=0}^n d(n,k)x^k, 
\quad n=0,1,2,\dots
\]
are called the \defn{Delannoy polynomials}.
The last step of a Delannoy path gives the recurrence
$$
d(n,k)=d(n-1,k-1)+d(n-1,k)+d(n-2,k-1)
$$
with $d(0,0)=1$ and $d(n,k)=0$ unless $0\le k\le n$.
Consequently,
\[
d_n(x)=(x+1)d_{n-1}(x)+xd_{n-2}(x)
\]
with $d_0(x)=1$ and $d_1(x)=x+1$.
It is known that the Delannoy polynomials are $\ga$-positive
(see~\cite{WZC19} for a combinatorial proof and~\cite{CKS09} for a general result and an algebraic proof).
Here a polynomial $f(x)$ with real coefficients is called \defn{$\ga$-positive} if it can be written as
\[
f(x)=\sum_{i=0}^{\floor{n/2}} \ga_i x^i (1+x)^{n-2i}
\]
with $\ga_i\ge0$ for all $i$.
Gamma positive polynomials arise often in combinatorics and geometry.
We refer the reader to the survey articles \cite{Bra15,Ath18}
and \cite{AS12,LXZ22,JL24,FY25,YYL26} for various examples and generalizations.

One structural source of $\ga$-positive polynomials is the rank-generating functions of graded posets.
If a poset has a symmetric Boolean decomposition,
then its rank-generating function is $\ga$-positive (see \S3 for details).
Classical examples include the decomposition of the noncrossing partition lattice by Simion and Ullman~\cite{SU91}
and that of the shard intersection order by Petersen~\cite{Pet13},
showing the $\ga$-positivity of the Narayana polynomials and the Eulerian polynomials, respectively.
Very recently, Farley and Srinivasan~\cite{FS25} constructed an explicit decomposition of the subspace lattice.
These results motivate an analogous study of the Delannoy polynomials.

The objective of this paper is to show that the rank-generating function of the Delannoy lattice $\cD_n$ is the Delannoy polynomial $d_n(x)$,
and to construct a symmetric Boolean decomposition of $\cD_n$.
We obtain that
\[
d_n(x)=\sum_{i=0}^{\floor{n/2}} \binom{n-i}{i} x^i(1+x)^{n-2i},
\]
which recovers the $\ga$-positivity of the Delannoy polynomials from the viewpoint of poset structure.
The rest of the paper is organized as follows.
In the next section, we prove that $\cD_n$ is a self-dual lattice whose rank-generating function is $d_n(x)$.
An algorithm is provided to find the join and meet of two arbitrary Delannoy paths.
In Section 3, we first construct a symmetric Boolean decomposition of $\cD_n$,
then refine it to a symmetric chain decomposition with the chain cover property and determine all maximum antichains.
In Section 4, we prove that $\cD_n$ is supersolvable.
As a consequence, its M\"obius function equals a signed Fibonacci number.
Section 5 presents explicit formulas for the characteristic polynomials and zeta polynomials of $\cD_n$.
Moreover, we show that the characteristic polynomials have only real zeros and their zeros interlace.
Finally, in the last section, we propose a related problem regarding the normalized matching property.

\section{The Delannoy lattice}
We begin with the poset terminology used throughout the paper and refer the reader to Stanley \cite[Chapter~3]{Sta12} for further background.
Let $P$ be a finite poset.
An element $t$ \defn{covers} $s$, denoted $s\cover t$,
if $s<t$ and no element $u\in P$ satisfies $s<u<t$.
A finite poset is completely determined by its cover relations.
We say that $P$ is \defn{graded of rank} $n$
if every maximal chain of $P$ has the same length $n$.
In this case there is a unique \defn{rank function} $r$: $P\to\{0,1,\dots,n\}$
such that $r(s)=0$ if $s$ is a minimal element of $P$,
and $r(t)=r(s)+1$ if $s\cover t$.
The \defn{rank-generating function} of $P$ is defined as
\[
F(P,x)=\sum_{s\in P}x^{r(s)}
      =\sum_{k=0}^n |P_k|x^k,
\]
where
\[
P_k:=\{s\in P\mid r(s)=k\},
\qquad 0\le k\le n.
\]
A fundamental example of graded posets is the Boolean algebra $\cB_n$ which consists of the subsets of $\{1,2,\dots,n\}$ ordered by inclusion.
Clearly, $F(\cB_n,x)=(1+x)^n$.

\begin{prop}\label{prop-poset}
    The poset $\cD_n$ is graded of rank $n$ with minimum element $H^n$ and maximum element $V^n$,
    and the rank of $p\in\cD_n$ is $v(p)+d(p)$.
    Moreover, the elementary moves in \eqref{eq-cover} are precisely the cover relations of $\cD_n$.
\end{prop}

\begin{proof}
Given a path $p\in\cD_n$,
replacing every $D$ in $p$ by $HH$ and every $V$ by $H$ yields $H^n$. Reversing these operations gives a sequence of elementary moves from $H^n$ to $p$. Similarly, replacing every $D$ by $VV$ and every $H$ by $V$ transforms $p$ into $V^n$.
Hence, $H^n\le p\le V^n$.
Therefore, $H^n$ and $V^n$ are the minimum and maximum elements of $\cD_n$, respectively.

Suppose that $p\to q$ is an elementary move.
Then $v(q)+d(q)=v(p)+d(p)+1$.
If $p<u<q$,
then the strict increase of $v+d$ along the partial order would give
\[
v(p)+d(p)<v(u)+d(u)<v(q)+d(q),
\]
which is impossible.
Hence, $p\cover q$.
Conversely, if $p\cover q$, consider a sequence
\[
p=p_0\to p_1\to\cdots\to p_{\ell}=q
\]
with $\ell\ge1$.
If $\ell\ge2$, then $p<p_1<q$, contradicting $p\cover q$.
Thus $\ell=1$, and every cover relation is an elementary move.
It follows that $r(p)=v(p)+d(p)$ is the rank function,
and every maximal chain of $\cD_n$ has the same length $n$,
since $r(V^n)=n$.
Therefore, $\cD_n$ is a graded poset of rank $n$.
\end{proof}

See Figure~\ref{fig-Dn} for the posets $\cD_1$, $\cD_2$, and $\cD_3$.

\begin{figure}[htbp]
\centering
\tikzset{
    cover/.style={
        draw=blue!70!black,
        line width=1pt,
        shorten <=2pt,
        shorten >=2pt
    },
    pathline/.style={
        draw=black,
        line width=1pt,
        line cap=round,
        line join=round
    },
    dot/.style={
        circle,
        fill=black,
        inner sep=1.5pt
    },
    pathbox/.style={
        rectangle,
        draw=none,
        fill=none,
        inner sep=7pt
    }
}

\tikzset{
pics/H/.style={code={
    \draw[pathline] (-.35,0)--(.35,0);
    \node[dot] at (-.35,0) {};
    \node[dot] at (.35,0) {};
}},
pics/V/.style={code={
    \draw[pathline] (0,-.35)--(0,.35);
    \node[dot] at (0,-.35) {};
    \node[dot] at (0,.35) {};
}},
pics/D/.style={code={
    \draw[pathline] (-.35,-.35)--(.35,.35);
    \node[dot] at (-.35,-.35) {};
    \node[dot] at (.35,.35) {};
}},
pics/HH/.style={code={
    \draw[pathline] (-.7,0)--(0,0)--(.7,0);
    \node[dot] at (-.7,0) {};
    \node[dot] at (0,0) {};
    \node[dot] at (.7,0) {};
}},
pics/HV/.style={code={
    \draw[pathline] (-.35,-.35)--(.35,-.35)--(.35,.35);
    \node[dot] at (-.35,-.35) {};
    \node[dot] at (.35,-.35) {};
    \node[dot] at (.35,.35) {};
}},
pics/VH/.style={code={
    \draw[pathline] (-.35,-.35)--(-.35,.35)--(.35,.35);
    \node[dot] at (-.35,-.35) {};
    \node[dot] at (-.35,.35) {};
    \node[dot] at (.35,.35) {};
}},
pics/VV/.style={code={
    \draw[pathline] (0,-.7)--(0,0)--(0,.7);
    \node[dot] at (0,-.7) {};
    \node[dot] at (0,0) {};
    \node[dot] at (0,.7) {};
}},
pics/HHH/.style={code={
    \draw[pathline] (-1.05,0)--(-.35,0)--(.35,0)--(1.05,0);
    \foreach \x in {-1.05,-.35,.35,1.05}
        \node[dot] at (\x,0) {};
}},
pics/HHV/.style={code={
    \draw[pathline] (-.7,-.35)--(0,-.35)--(.7,-.35)--(.7,.35);
    \node[dot] at (-.7,-.35) {};
    \node[dot] at (0,-.35) {};
    \node[dot] at (.7,-.35) {};
    \node[dot] at (.7,.35) {};
}},
pics/HD/.style={code={
    \draw[pathline] (-.7,-.35)--(0,-.35)--(.7,.35);
    \node[dot] at (-.7,-.35) {};
    \node[dot] at (0,-.35) {};
    \node[dot] at (.7,.35) {};
}},
pics/HVH/.style={code={
    \draw[pathline] (-.7,-.35)--(0,-.35)--(0,.35)--(.7,.35);
    \node[dot] at (-.7,-.35) {};
    \node[dot] at (0,-.35) {};
    \node[dot] at (0,.35) {};
    \node[dot] at (.7,.35) {};
}},
pics/DH/.style={code={
    \draw[pathline] (-.7,-.35)--(0,.35)--(.7,.35);
    \node[dot] at (-.7,-.35) {};
    \node[dot] at (0,.35) {};
    \node[dot] at (.7,.35) {};
}},
pics/VHH/.style={code={
    \draw[pathline] (-.7,-.35)--(-.7,.35)--(0,.35)--(.7,.35);
    \node[dot] at (-.7,-.35) {};
    \node[dot] at (-.7,.35) {};
    \node[dot] at (0,.35) {};
    \node[dot] at (.7,.35) {};
}},
pics/HVV/.style={code={
    \draw[pathline] (-.35,-.7)--(.35,-.7)--(.35,0)--(.35,.7);
    \node[dot] at (-.35,-.7) {};
    \node[dot] at (.35,-.7) {};
    \node[dot] at (.35,0) {};
    \node[dot] at (.35,.7) {};
}},
pics/DV/.style={code={
    \draw[pathline] (-.35,-.7)--(.35,0)--(.35,.7);
    \node[dot] at (-.35,-.7) {};
    \node[dot] at (.35,0) {};
    \node[dot] at (.35,.7) {};
}},
pics/VHV/.style={code={
    \draw[pathline] (-.35,-.7)--(-.35,0)--(.35,0)--(.35,.7);
    \node[dot] at (-.35,-.7) {};
    \node[dot] at (-.35,0) {};
    \node[dot] at (.35,0) {};
    \node[dot] at (.35,.7) {};
}},
pics/VD/.style={code={
    \draw[pathline] (-.35,-.7)--(-.35,0)--(.35,.7);
    \node[dot] at (-.35,-.7) {};
    \node[dot] at (-.35,0) {};
    \node[dot] at (.35,.7) {};
}},
pics/VVH/.style={code={
    \draw[pathline] (-.35,-.7)--(-.35,0)--(-.35,.7)--(.35,.7);
    \node[dot] at (-.35,-.7) {};
    \node[dot] at (-.35,0) {};
    \node[dot] at (-.35,.7) {};
    \node[dot] at (.35,.7) {};
}},
pics/VVV/.style={code={
    \draw[pathline] (0,-1.05)--(0,-.35)--(0,.35)--(0,1.05);
    \foreach \y in {-1.05,-.35,.35,1.05}
        \node[dot] at (0,\y) {};
}}
}

\newcommand{\dpathnode}[3]{%
    \node[pathbox] (#1) at #2 {%
        \begin{tikzpicture}[baseline=-0.5ex]
            \pic {#3};
        \end{tikzpicture}%
    };%
}
\scalebox{0.7}{
\begin{tikzpicture}[scale=1]

\begin{scope}[xshift=-1cm]
    \dpathnode{A0}{(0,0)}{H}
    \dpathnode{A1}{(0,2.0)}{V}

    \draw[cover] (A0.north)--(A1.south);

    \node at (0,-0.85) {\Large$\cD_1$};
\end{scope}

\begin{scope}[xshift=3cm]
    \dpathnode{B0}{(0,0)}{HH}

    \dpathnode{B1L}{(-1.7,2.0)}{HV}
    \dpathnode{B1M}{(0,2.0)}{VH}
    \dpathnode{B1R}{(1.7,2.0)}{D}

    \dpathnode{B2}{(0,4.5)}{VV}

    \draw[cover] (B0.north)--(B1L.south);
    \draw[cover] (B0.north)--(B1M.south);
    \draw[cover] (B0.north)--(B1R.south);

    \draw[cover] (B1L.north)--(B2.south);
    \draw[cover] (B1M.north)--(B2.south);
    \draw[cover] (B1R.north)--(B2.south);

    \node at (0,-0.85) {\Large$\cD_2$};
\end{scope}

\begin{scope}[xshift=12cm]
    \dpathnode{C0}{(0,0)}{HHH}

    \dpathnode{C1a}{(-4.8,2.4)}{HHV}
    \dpathnode{C1b}{(-2.4,2.4)}{HVH}
    \dpathnode{C1c}{(0,2.4)}{VHH}
    \dpathnode{C1d}{(2.4,2.4)}{HD}
    \dpathnode{C1e}{(4.8,2.4)}{DH}

    \dpathnode{C2a}{(-4.8,5.2)}{HVV}
    \dpathnode{C2b}{(-2.4,5.2)}{VHV}
    \dpathnode{C2c}{(0,5.2)}{VVH}
    \dpathnode{C2d}{(2.4,5.2)}{DV}
    \dpathnode{C2e}{(4.8,5.2)}{VD}

    \dpathnode{C3}{(0,8.5)}{VVV}

    \draw[cover] (C0.north)--(C1a.south);
    \draw[cover] (C0.north)--(C1b.south);
    \draw[cover] (C0.north)--(C1c.south);
    \draw[cover] (C0.north)--(C1d.south);
    \draw[cover] (C0.north)--(C1e.south);

    \draw[cover] (C1a.north)--(C2a.south); 
    \draw[cover] (C1a.north)--(C2b.south); 
    \draw[cover] (C1a.north)--(C2d.south); 

    \draw[cover] (C1b.north)--(C2a.south); 
    \draw[cover] (C1b.north)--(C2c.south); 

    \draw[cover] (C1c.north)--(C2b.south); 
    \draw[cover] (C1c.north)--(C2c.south); 
    \draw[cover] (C1c.north)--(C2e.south); 

    \draw[cover] (C1d.north)--(C2a.south); 
    \draw[cover] (C1d.north)--(C2e.south); 

    \draw[cover] (C1e.north)--(C2c.south); 
    \draw[cover] (C1e.north)--(C2d.south); 

    \draw[cover] (C2a.north)--(C3.south);
    \draw[cover] (C2b.north)--(C3.south);
    \draw[cover] (C2c.north)--(C3.south);
    \draw[cover] (C2d.north)--(C3.south);
    \draw[cover] (C2e.north)--(C3.south);

    \node at (0,-0.85) {\Large$\cD_3$};
\end{scope}

\end{tikzpicture}
}
\caption{The posets $\cD_1$, $\cD_2$, and $\cD_3$.}
\label{fig-Dn}
\end{figure}

\begin{prop}\label{prop:rank-generating}
    The rank-generating function of $\cD_n$ is the Delannoy polynomial $d_n(x)$.
\end{prop}

\begin{proof}
Note that each $p\in\cD_n$ ends at a point with vertical coordinate $v(p)+d(p)$,
which is precisely the rank of $p$.
Hence, the elements of rank $k$ in $\cD_n$ are exactly the Delannoy paths from $(0,0)$ to $(n-k,k)$.
There are $d(n,k)$ such paths. 
Therefore,
\[
F(\cD_n,x)=\sum_{k=0}^n d(n,k)x^k=d_n(x)
\]
as desired.
\end{proof}



\begin{prop}\label{prop:self-dual}
The poset $\cD_n$ is self-dual.
That is, there exists a bijection $\sig: \cD_n\to\cD_n$ such that $p\le q$ if and only if $\sig(q)\le\sig(p)$.
\end{prop}

\begin{proof}
Define $\sig: \cD_n\to\cD_n$ by interchanging $H$ and $V$ and fixing $D$.
It is clear that $\sig$ is an involution and it maps the three elementary moves in \eqref{eq-cover} to
\[
aVb \to aHb,\qquad
aDb \to aHHb,\qquad
aVVb \to aDb.
\]
Equivalently, $p\to q$ if and only if $\sig(q)\to\sig(p)$.
Taking transitive closures gives that $p\le q$ if and only if $\sig(q)\le\sig(p)$.
Hence, $\cD_n$ is self-dual.
\end{proof}

Let $p,q\in P$.
The \defn{join} of $p$ and $q$, denoted $p\join q$, is their least upper bound.
Dually, the \defn{meet} of $p$ and $q$, denoted $p\meet q$, is their greatest lower bound.
A poset is called a \defn{lattice} if every pair of elements has a join and a meet.
It is well known that the Boolean algebra $\cB_n$ is a lattice, called the \defn{Boolean lattice}.
In the following we show that $\cD_n$ is also a lattice,
which we call the \defn{Delannoy lattice}.

\begin{thm}\label{thm:Dn-lattice}
The poset $\cD_n$ is a lattice.
\end{thm}

\begin{proof}
It suffices to prove that every pair of paths in $\cD_n$ has a join.
Then by self-duality, every pair also has a meet.
Let $p_1, p_2\in \cD_n$.
If $p_1=p_2$, then $p_1\join p_2=p_1$.
Suppose that $p_1$ and $p_2$ are distinct.
Let $a$ be their longest common prefix.
We define an operator $\T$ on $\cD_n\times\cD_n$ as one of the following cases
(interchanging $p_1$ and $p_2$ if necessary)
\begin{equation}\label{eq-T}
    \begin{array}{c|c|c}
    \text{Case} & (p_1,p_2) & \T(p_1,p_2) \\
    \hline
    \rm(i) & (aHb_1, aVb_2) & (aVb_1, aVb_2) \\
    \rm(ii) & (aHHb_1, aDb_2) & (aDb_1, aDb_2) \\
    \rm(iii) & (aHb_1, aDb_2) & (aVb_1, aVVb_2) \\
    \rm(iv) & (aVb_1, aDb_2) & (aVb_1, aVVb_2)
    \end{array}
\end{equation}
Here Case (ii) has priority over Case (iii),
that is, in Case (iii) the word $b_1$ does not begin with $H$.
This defines $\T$ symmetrically on ordered pairs by restoring the original order if the two paths were interchanged.
We proceed by iterating $\T$ and denote
\[
(p_1^{(i)},p_2^{(i)}):=\T^i(p_1,p_2).
\]
Observe that the common prefix length of $p_1^{(i)}$ and $p_2^{(i)}$ strictly increases with $i$.
Since a path in $\cD_n$ has at most $n$ steps,
there exists a smallest $k$ such that $p_1^{(k)}=p_2^{(k)}$.

\textbf{Claim}.
Let $u:=p_1^{(k)}=p_2^{(k)}$.
Then $p_1\join p_2=u$.

We need to show that $u$ is the least common upper bound of $p_1$ and $p_2$.

For any $p,q\in\cD_n$, let
\[
\cU(p)=\{q\in\cD_n\mid p\le q\}
\]
denote the set of upper bounds of $p$ and
\[
\UB(p,q):=\cU(p)\cap\cU(q)
\]
be the set of common upper bounds of $p$ and $q$.
If $\cS$ is a set of words and $w$ is a word, write
\[
w\cS=\{ws\mid s\in\cS\}.
\]
The elementary moves give
\begin{align}\label{eq-U}\nonumber
    \cU(Hb)&=H\cU(b)\cup V\cU(b)\cup
    \begin{cases}
        D\cU(c), & \text{if $b=Hc$},\\
        \emptyset, & \text{otherwise},
    \end{cases}\\
    \cU(Vb)&=V\cU(b),\\ \nonumber
    \cU(Db)&=D\cU(b)\cup VV\cU(b).
\end{align}
It is not hard to verify that
\begin{enumerate}[(i)]
    \item $\cU(Hb_1)\cap\cU(Vb_2)=\cU(Vb_1)\cap\cU(Vb_2)$;
    \item $\cU(HHb_1)\cap\cU(Db_2)=\cU(Db_1)\cap\cU(Db_2)$;
    \item $\cU(Hb_1)\cap\cU(Db_2)=\cU(Vb_1)\cap\cU(VVb_2)$, where $b_1$ does not begin with $H$;
    \item $\cU(Vb_1)\cap\cU(Db_2)=\cU(Vb_1)\cap\cU(VVb_2)$.
\end{enumerate}
By applying each case to the corresponding one in \eqref{eq-T}, we obtain
\begin{equation}\label{eq-UB}
    \UB(p_1,p_2)=\UB(p_1^{(1)},p_2^{(1)}).
\end{equation}
Repeated application of \eqref{eq-UB} yields
\[
\UB(p_1,p_2)=\UB(p_1^{(k)},p_2^{(k)})=\UB(u,u)=\cU(u).
\]
This implies that $u$ is a common upper bound of $p_1$ and $p_2$,
and every common upper bound lie above $u$.
Thus, $u$ is the least common upper bound of $p_1$ and $p_2$, as desired.
\end{proof}

\begin{exa}
    Let $p_1=HDHHHV$, $p_2=DVHDH\in\cD_7$.
    The procedure for finding $p_1\join p_2$ is shown in Figure \ref{fig-join},
    where the red steps mark the first distinguishable step(s) of $p_1^{(i)}$ and $p_2^{(i)}$,
    and the roman numbers above each $\mapsto$ indicate which case of \eqref{eq-T} applies when operating $\T$.
    Then we derive $p_1\join p_2=p_1^{(4)}=p_2^{(4)}$.
\end{exa}

\begin{figure}[htbp]
    \centering
    \begin{tabular}{*{4}{cc}l}
        $p_1$ && $p_1^{(1)}$ && $p_1^{(2)}$ && $p_1^{(3)}$ && $p_1^{(4)}$ \\[5pt]
        \begin{tikzpicture}[scale=0.45]
            \draw[thick] (0,0)--(1,0)--(2,1)--(3,1)--(4,1)--(5,1)--(5,2);
            \draw[red,thick] (0,0)--(1,0);
            \foreach \x/\y in {0/0,1/0,2/1,3/1,4/1,5/1,5/2}
                \filldraw (\x,\y) circle(.1);
        \end{tikzpicture}
        &&
        \begin{tikzpicture}[scale=0.45]
            \draw[thick] (0,0)--(0,1)--(1,2)--(2,2)--(3,2)--(4,2)--(4,3);
            \draw[red,thick] (0,1)--(1,2);
            \foreach \x/\y in {0/0,0/1,1/2,2/2,3/2,4/2,4/3}
                \filldraw (\x,\y) circle(.1);
        \end{tikzpicture}
        &&
        \begin{tikzpicture}[scale=0.45]
            \draw[thick] (0,0)--(0,1)--(0,2)--(0,3)--(1,3)--(2,3)--(3,3)--(3,4);
            \draw[red,thick] (1,3)--(3,3);
            \foreach \x/\y in {0/0,0/1,0/2,0/3,1/3,2/3,3/3,3/4}
                \filldraw (\x,\y) circle(.1);
        \end{tikzpicture}
        &&
        \begin{tikzpicture}[scale=0.45]
            \draw[thick] (0,0)--(0,1)--(0,2)--(0,3)--(1,3)--(2,4)--(2,5);
            \draw[red,thick] (2,4)--(2,5);
            \foreach \x/\y in {0/0,0/1,0/2,0/3,1/3,2/4,2/5}
                \filldraw (\x,\y) circle(.1);
        \end{tikzpicture}
        &&
        \begin{tikzpicture}[scale=0.45]
            \draw[thick] (0,0)--(0,1)--(0,2)--(0,3)--(1,3)--(2,4)--(2,5);
            \foreach \x/\y in {0/0,0/1,0/2,0/3,1/3,2/4,2/5}
                \filldraw (\x,\y) circle(.1);
        \end{tikzpicture}\\
        & $\xmapsto{\rm(iii)}$ && $\xmapsto{\rm(iv)}$ && $\xmapsto{\rm(ii)}$ && $\xmapsto{\rm(i)}$ \\
        \begin{tikzpicture}[scale=0.45]
            \draw[thick] (0,0)--(1,1)--(1,2)--(2,2)--(3,3)--(4,3);
            \draw[red,thick] (0,0)--(1,1);
            \foreach \x/\y in {0/0,1/1,1/2,2/2,3/3,4/3}
                \filldraw (\x,\y) circle(.1);
        \end{tikzpicture}
        &&
        \begin{tikzpicture}[scale=0.45]
            \draw[thick] (0,0)--(0,1)--(0,2)--(0,3)--(1,3)--(2,4)--(3,4);
            \draw[red,thick] (0,1)--(0,2);
            \foreach \x/\y in {0/0,0/1,0/2,0/3,1/3,2/4,3/4}
                \filldraw (\x,\y) circle(.1);
        \end{tikzpicture}
        &&
        \begin{tikzpicture}[scale=0.45]
            \draw[thick] (0,0)--(0,1)--(0,2)--(0,3)--(1,3)--(2,4)--(3,4);
            \draw[red,thick] (1,3)--(2,4);
            \foreach \x/\y in {0/0,0/1,0/2,0/3,1/3,2/4,3/4}
                \filldraw (\x,\y) circle(.1);
        \end{tikzpicture}
        &&
        \begin{tikzpicture}[scale=0.45]
            \draw[thick] (0,0)--(0,1)--(0,2)--(0,3)--(1,3)--(2,4)--(3,4);
            \draw[red,thick] (2,4)--(3,4);
            \foreach \x/\y in {0/0,0/1,0/2,0/3,1/3,2/4,3/4}
                \filldraw (\x,\y) circle(.1);
        \end{tikzpicture}
        &&
        \begin{tikzpicture}[scale=0.45]
            \draw[thick] (0,0)--(0,1)--(0,2)--(0,3)--(1,3)--(2,4)--(2,5);
            \foreach \x/\y in {0/0,0/1,0/2,0/3,1/3,2/4,2/5}
                \filldraw (\x,\y) circle(.1);
        \end{tikzpicture}\\
        $p_2$ && $p_2^{(1)}$ && $p_2^{(2)}$ && $p_2^{(3)}$ && $p_2^{(4)}$
    \end{tabular}
    \caption{Procedure for finding $p_1\join p_2$ with $p_1=HDHHHV$, $p_2=DVHDH\in\cD_7$.}
    \label{fig-join}
\end{figure}

\begin{rem}
    Orders on Delannoy paths have previously been studied from a different viewpoint.
    Autebert et al. \cite{ALS02} considered paths between two fixed endpoints to obtain a distributive lattice.
    Autebert and Schwer \cite{AS03} extended this approach to generalized Delannoy path.
    These constructions are related to Greene's posets of shuffles \cite{Gre88} and their generalizations \cite{Her02}.
    The Delannoy lattice $\cD_n$ studied in this paper is different;
    its paths have varying endpoints on the line $x+y=n$.
    Most importantly, the rank-generating function of our Delannoy lattice is the Delannoy polynomial, which is not the case for the lattice of Autebert et al.
    In particular, $\cD_n$ is not distributive for $n\ge2$
    (see Remark \ref{rem-notdis}).
\end{rem}

\section{Symmetric decompositions}

We first establish a symmetric Boolean decomposition of $\cD_n$,
then present an algorithm for constructing a symmetric chain decomposition and for determining the maximum antichains of $\cD_n$.

\subsection{A symmetric Boolean decomposition}

\begin{defi}
Let $P$ be a finite graded poset of rank $n$ with rank function $r$.
A subset $B\subseteq P$ is called a \defn{symmetric Boolean subset} if under the induced order,
\begin{enumerate}[(1)]
    \item $B$ has a bottom element $b$ and a top element $t$;
    \item $r(b)+r(t)=n$;
    \item $B$ is isomorphic to the Boolean lattice $\cB_{r(t)-r(b)}$.
\end{enumerate}
A \defn{symmetric Boolean decomposition} of $P$ is a partition of $P$ into pairwise disjoint symmetric Boolean subsets. 
In this case, $P$ is called a \defn{symmetric Boolean order}.
\end{defi}

If a symmetric Boolean subset has bottom rank $i$,
then its top rank is $n-i$,
so it is isomorphic to $\cB_{n-2i}$ and contributes $x^i(1+x)^{n-2i}$ to the rank-generating function.
Hence, if $P$ is a symmetric Boolean order, then $F(P,x)$ is $\gamma$-positive.
More precisely,
\[
F(P,x)=\sum_{i=0}^{\floor{n/2}} \ga_i x^i(1+x)^{n-2i},
\]
where $\gamma_i$ counts symmetric Boolean subsets with bottom rank $i$.

In order to construct a symmetric Boolean decomposition of $\cD_n$,
we introduce the following encoding by matchings of a path graph.
Given $p\in\cD_n$,
consider a path graph $G(p)$ with vertices $1,2,\dots,n$ in order from left to right,
and place $H$, $V$, and $D$ on the vertices such that
$H$ or $V$ occupies one vertex, whereas $D$ occupies two consecutive vertices.
Let $M(p)$ be the set of edges in $G(p)$ whose endpoints are both occupied by the same $D$-step of $p$.
Then $M(p)$ forms a matching of $G(p)$.
A vertex is called a \defn{free position} of $p$ if it is not covered by $M(p)$.
So the number of free positions of $p$ is $h(p)+v(p)$.
Let
\(
G_n=\{G(p)\mid p\in\cD_n\}
\)
and let $\cM_n$ be the set of all matchings of $G_n$.

\begin{exa}
    Let $p=HDDVDH\in\cD_9$.
    Then $G(p)$ is as follows and $M(p)=\{23,45,78\}$.
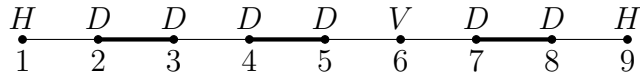
    \begin{center}
        \begin{tikzpicture}
            \draw (0,0)--(8,0);
            \foreach \x in {1,...,9}
                {
                    \filldraw (\x-1,0) circle(.05);
                    \node[below] at (\x-1,0) {$\x$};
                }
            \foreach \x in {1,9}
                \node[above] at (\x-1,0) {$H$};
            \node[above] at (5,0) {$V$};
            \foreach \x in {2,3,4,5,7,8}
                \node[above] at (\x-1,0) {$D$};
            \draw[ultra thick] (1,0)--(2,0) (3,0)--(4,0) (6,0)--(7,0);
        \end{tikzpicture}
    \end{center}
\end{exa}

\begin{thm}\label{sbo}
    The Delannoy lattice $\cD_n$ is a symmetric Boolean order.
    More precisely,
    for each $M\in\cM_n$, let
    \[
    B_M=\{p\in\cD_n\mid M(p)=M\}.
    \]
    Then the sets $B_M$, where $M$ ranges over $\cM_n$, form a symmetric Boolean decompostion of $\cD_n$.
    Moreover,
    \[
    \cD_n=\biguplus_{M\in\cM_n} B_M=\biguplus_{i=0}^{\floor{n/2}} \binom{n-i}{i} \cB_{i,n-i},
    \]
    where $\cB_{i,n-i}$ denotes the Boolean lattice,
    formed by $p\in B_M$ with $|M|=i$,
    whose bottom and top elements have ranks $i$ and $n-i$, respectively.
\end{thm}

\begin{proof}
Every path $p\in\cD_n$ has a unique matching $M(p)$,
so the sets $B_M$ form a partition of $\cD_n$.
Each matching $M\in\cM_n$ of size $i$ contributes $2i$ copies of $D$,
leaving $n-2i$ free positions.
Then a path in $B_M$ is obtained by filling each free position independently with $H$ or $V$.
Note that the Boolean lattice $\cB_n$ can alternatively be regarded as the set of lattice paths from $(0,0)$ to the line $x+y=n$ using steps $H$ and $V$,
ordered by the cover relation $H\cover V$.
Thus, $B_M$, with its induced order, is isomorphic to $\cB_{n-2i}$.

The minimum element of $B_M$ has all free positions filled with $H$,
and hence has rank $i$.
The maximum element has all free positions filled with $V$,
and hence has rank $i+(n-2i)=n-i$.
Therefore, $B_M$ is symmetric.

It remains to count the matchings.
A matching of size $i$ in $G_n$ is determined by left endpoints choosing from $1,2,\dots,n-1$ that are not consecutive.
Thus there are $\binom{n-i}{i}$ matchings of size $i$,
which proves the desired decomposition.
See Figure \ref{fig:D3-SBO} for instance.
\end{proof}

\begin{figure}[htbp]
\centering
\scalebox{0.7}{
\begin{tikzpicture}[
    scale=0.95,
    line cap=round,
    line join=round,
    cover/.style={
        draw=blue!70!black,
        line width=0.8pt,
        shorten <=2pt,
        shorten >=2pt
    },
    redcover/.style={
        draw=red!85!black,
        line width=1.8pt,
        shorten <=2pt,
        shorten >=2pt
    },
    pathline/.style={
        draw=black,
        line width=1pt,
        line cap=round,
        line join=round
    },
    dot/.style={
        circle,
        fill=black,
        inner sep=1.5pt
    },
    pathbox/.style={
        rectangle,
        draw=none,
        fill=none,
        inner sep=7pt
    }
]

\tikzset{
pics/HHH/.style={code={
    \draw[pathline] (-1.05,0)--(-.35,0)--(.35,0)--(1.05,0);
    \foreach \x in {-1.05,-.35,.35,1.05}
        \node[dot] at (\x,0) {};
}},
pics/VVV/.style={code={
    \draw[pathline] (0,-1.05)--(0,-.35)--(0,.35)--(0,1.05);
    \foreach \y in {-1.05,-.35,.35,1.05}
        \node[dot] at (0,\y) {};
}},
pics/HHV/.style={code={
    \draw[pathline] (-.7,-.35)--(0,-.35)--(.7,-.35)--(.7,.35);
    \foreach \x/\y in {-.7/-.35,0/-.35,.7/-.35,.7/.35}
        \node[dot] at (\x,\y) {};
}},
pics/HD/.style={code={
    \draw[pathline] (-.7,-.35)--(0,-.35)--(.7,.35);
    \foreach \x/\y in {-.7/-.35,0/-.35,.7/.35}
        \node[dot] at (\x,\y) {};
}},
pics/HVH/.style={code={
    \draw[pathline] (-.7,-.35)--(0,-.35)--(0,.35)--(.7,.35);
    \foreach \x/\y in {-.7/-.35,0/-.35,0/.35,.7/.35}
        \node[dot] at (\x,\y) {};
}},
pics/DH/.style={code={
    \draw[pathline] (-.7,-.35)--(0,.35)--(.7,.35);
    \foreach \x/\y in {-.7/-.35,0/.35,.7/.35}
        \node[dot] at (\x,\y) {};
}},
pics/VHH/.style={code={
    \draw[pathline] (-.7,-.35)--(-.7,.35)--(0,.35)--(.7,.35);
    \foreach \x/\y in {-.7/-.35,-.7/.35,0/.35,.7/.35}
        \node[dot] at (\x,\y) {};
}},
pics/HVV/.style={code={
    \draw[pathline] (-.35,-.7)--(.35,-.7)--(.35,0)--(.35,.7);
    \foreach \x/\y in {-.35/-.7,.35/-.7,.35/0,.35/.7}
        \node[dot] at (\x,\y) {};
}},
pics/DV/.style={code={
    \draw[pathline] (-.35,-.7)--(.35,0)--(.35,.7);
    \foreach \x/\y in {-.35/-.7,.35/0,.35/.7}
        \node[dot] at (\x,\y) {};
}},
pics/VHV/.style={code={
    \draw[pathline] (-.35,-.7)--(-.35,0)--(.35,0)--(.35,.7);
    \foreach \x/\y in {-.35/-.7,-.35/0,.35/0,.35/.7}
        \node[dot] at (\x,\y) {};
}},
pics/VD/.style={code={
    \draw[pathline] (-.35,-.7)--(-.35,0)--(.35,.7);
    \foreach \x/\y in {-.35/-.7,-.35/0,.35/.7}
        \node[dot] at (\x,\y) {};
}},
pics/VVH/.style={code={
    \draw[pathline] (-.35,-.7)--(-.35,0)--(-.35,.7)--(.35,.7);
    \foreach \x/\y in {-.35/-.7,-.35/0,-.35/.7,.35/.7}
        \node[dot] at (\x,\y) {};
}}
}

\def\dpathnode#1#2#3{%
    \node[pathbox] (#1) at #2 {%
        \begin{tikzpicture}[baseline=-0.5ex]
            \pic {#3};
        \end{tikzpicture}%

    };%
}

\dpathnode{C0}{(0,0)}{HHH}

\dpathnode{C1a}{(-5.2,2.7)}{HHV}
\dpathnode{C1b}{(-2.6,2.7)}{HVH}
\dpathnode{C1c}{(0,2.7)}{VHH}
\dpathnode{C1d}{(2.6,2.7)}{HD}
\dpathnode{C1e}{(5.2,2.7)}{DH}

\dpathnode{C2a}{(-5.2,5.6)}{HVV}
\dpathnode{C2b}{(-2.6,5.6)}{VHV}
\dpathnode{C2c}{(0,5.6)}{VVH}
\dpathnode{C2d}{(2.6,5.6)}{DV}
\dpathnode{C2e}{(5.2,5.6)}{VD}

\dpathnode{C3}{(0,8.8)}{VVV}


\draw[cover] (C0.north)--(C1a.south); 
\draw[cover] (C0.north)--(C1b.south); 
\draw[cover] (C0.north)--(C1c.south); 
\draw[cover] (C0.north)--(C1d.south); 
\draw[cover] (C0.north)--(C1e.south); 

\draw[cover] (C1a.north)--(C2a.south); 
\draw[cover] (C1a.north)--(C2b.south); 
\draw[cover] (C1a.north)--(C2d.south); 

\draw[cover] (C1b.north)--(C2a.south); 
\draw[cover] (C1b.north)--(C2c.south); 

\draw[cover] (C1c.north)--(C2b.south); 
\draw[cover] (C1c.north)--(C2c.south); 
\draw[cover] (C1c.north)--(C2e.south); 

\draw[cover] (C1d.north)--(C2a.south); 
\draw[cover] (C1d.north)--(C2e.south); 

\draw[cover] (C1e.north)--(C2c.south); 
\draw[cover] (C1e.north)--(C2d.south); 

\draw[cover] (C2a.north)--(C3.south); 
\draw[cover] (C2b.north)--(C3.south); 
\draw[cover] (C2c.north)--(C3.south); 
\draw[cover] (C2d.north)--(C3.south); 
\draw[cover] (C2e.north)--(C3.south); 



\draw[redcover] (C0.north)--(C1a.south); 
\draw[redcover] (C0.north)--(C1b.south); 
\draw[redcover] (C0.north)--(C1c.south); 

\draw[redcover] (C1a.north)--(C2a.south); 
\draw[redcover] (C1a.north)--(C2b.south); 

\draw[redcover] (C1b.north)--(C2a.south); 
\draw[redcover] (C1b.north)--(C2c.south); 

\draw[redcover] (C1c.north)--(C2b.south); 
\draw[redcover] (C1c.north)--(C2c.south); 

\draw[redcover] (C2a.north)--(C3.south); 
\draw[redcover] (C2b.north)--(C3.south); 
\draw[redcover] (C2c.north)--(C3.south); 

\draw[redcover] (C1d.north)--(C2e.south); 
\draw[redcover] (C1e.north)--(C2d.south); 

\end{tikzpicture}
}
\caption{A symmetric Boolean decomposition of $\cD_3$, with the symmetric Boolean subsets highlighted in red. Precisely, $\cD_3=\cB_{0,3}\cup2\cB_{1,2}$.}
\label{fig:D3-SBO}
\end{figure}

Combining Proposition~\ref{prop:rank-generating} and Theorem~\ref{sbo} gives the following.

\begin{cor}
The Delannoy polynomials are $\gamma$-positive with
\[
d_n(x)=\sum_{i=0}^{\floor{n/2}} \binom{n-i}{i}x^i(1+x)^{n-2i}.
\]
\end{cor}

\subsection{A symmetric chain decomposition}

\begin{defi}
Let $P$ be a finite graded poset of rank $n$.
A saturated chain
\[ 
x_0\cover x_1\cover \cdots \cover x_{\ell} 
\] 
is \defn{symmetric} if $r(x_0)+r(x_{\ell})=n$.
A \defn{symmetric chain decomposition} of $P$ is a partition of $P$ into pairwise disjoint symmetric chains.
\end{defi}

It is known that the Boolean lattice admits a symmetric chain decomposition.
An explicit description of the symmetric chains was given by Greene and Kleitman~\cite{GK76}
using bracketing construction.
We apply the Greene-Kleitman decomposition separately to each Boolean component $B_M$ and give an explicit symmetric chain decomposition of $\cD_n$.

For $p=s_1s_2\cdots s_m\in\cD_n$,
define a sequence $x_1x_2\cdots x_m$ by
\[
x_i=
\begin{cases}
(, & \text{if } s_i=H,\\
), & \text{if } s_i=V,\\
*, & \text{if } s_i=D.
\end{cases}
\]
Fix each ``$*$'' in the sequence and scan the remaining letters from left to right,
matching each ``$)$'' with the nearest unmatched ``$($'' to its left.
After the matching is completed, the unmatched parentheses other than $*$ have the form
\begin{equation}\label{eq-matching}
    \underbrace{))\cdots))}_u \underbrace{((\cdots((}_{\ell-u}
\end{equation}
for unique integers $u$ and $\ell$ with $0\le u\le\ell$.
Keep all matched parentheses and all ``$*$'' fixed.
For $0\le k\le\ell$, let $p_k$ be the Delannoy path obtained by placing ``$)$'' in the first $k$ unmatched positions and ``$($'' in the remaining $\ell-k$ unmatched positions.
Then
\[
p_0\cover p_1\cover\cdots\cover p_{\ell}
\]
is the chain containg the original path $p$ (which is $p_u$).

\begin{exa}
    Let $p=HDVVDDVHHVDVHD\in\cD_{19}$.
    Then after the matching,
    \[
    p\mapsto \tre{(}*\tre{)})**)\tre{(()}*\tre{)}(*
    \]
    with $u=2$ and $\ell=3$,
    where matched parentheses are shown in red.
    So the following chian $p_0\cover p_1\cover p=p_2\cover p_3$ contains $p$.
    \[
    \begin{array}{rcc}
        p_0=HDVHDDHHHVDVHD & \mapsto & \tre{(}*\tre{)}(**(\tre{(()}*\tre{)}(* \\
        p_1=HDVVDDHHHVDVHD & \mapsto & \tre{(}*\tre{)})**(\tre{(()}*\tre{)}(* \\
        p=p_2=HDVVDDVHHVDVHD & \mapsto & \tre{(}*\tre{)})**)\tre{(()}*\tre{)}(* \\
        p_3=HDVVDDVHHVDVVD & \mapsto & \tre{(}*\tre{)})**)\tre{(()}*\tre{)})*
    \end{array}
    \]
\end{exa}

\begin{prop}\label{prop-SCD}
    The chains constructed above form a symmetric chain decomposition of $\cD_n$.
\end{prop}

\begin{proof}
Within a fiexed component $B_M$, the above construction is the Green-Kleitman decomposition of a Boolean lattice.
Note that changing the unmatched parentheses as above does not alter any of the matched pairs.
Thus every $p_k$ has the same matched pairs and the same unmatched positions,
and therefore,
applying the construction to any $p_k$ recovers the same chain.
So all these chains partition $B_M$, and hence, partition $\cD_n$.

It remains to check the symmetry of each chain with respect to the rank
(total number of $V$'s and $D$'s)
in $\cD_n$.
Suppose that $|M|=i$ and the matching procedure creates $t$ matched pairs.
Then \eqref{eq-matching} implies $2t+\ell+2i=n$.
Since each matched pairs contains one $V$,
and all unmatched positions in $p_0$ (resp. $p_{\ell}$) are $H$ (resp. $V$),
we have $r(p_0)=t+i$ and $r(p_k)=t+\ell+i$.
Hence,
\(
r(p_0)+r(p_k)=2t+\ell+2i=n.
\)
Therefore, every chain is symmetric.
\end{proof}

We next verify a compatibility property of this symmetric chain decomposition.
Following Griggs, Killian and Savage~\cite{GKS04}, we define the chain cover property.
Let $\sC$ be a symmetric chain decomposition of a finite graded poset.
A chain of maximum length in $\sC$ is called a \defn{root chain}.
We say that $\mathscr C$ has the \defn{chain cover property} if, for every non-root chain $C\in\mathscr C$, there exists a chain $C^+\in\mathscr C$ such that
\[
\min C^+\cover\min C
\qquad\text{and}\qquad
\max C\cover\max C^+.
\]

\begin{lem}[\cite{GK76,GKS04}]
\label{lem:GK-chain-cover}
The Greene-Kleitman symmetric chain decomposition of the Boolean lattice $\cB_n$ has the chain cover property.
Its unique root chain begins at $\emptyset$ and ends at $[n]$.
\end{lem}

\begin{thm}
\label{thm:Dn-chain-cover}
The symmetric chain decomposition of $\cD_n$ in Proposition \ref{prop-SCD} has the chain cover property.
\end{thm}

\begin{proof}
Let $C$ be a non-root chain contained in a Boolean component $B_M$ of $\cD_n$.
If $C$ is not the root chain of the Greene-Kleitman decomposition of $B_M$,
then Lemma \ref{lem:GK-chain-cover} provides the required chain $C^+$ inside the same component.

Now suppose that $C$ is the root chain of the Greene-Kleitman decomposition of $B_M$.
The only component of Boolean rank $n$ is $B_{\emptyset}$.
Its root chain has length $n$, so it is also a root chain of the whole decomposition.
This implies that $M\neq\emptyset$.
Choose an edge $e\in M$ and set $M^+=M\backslash\{e\}$.
Let $C^+$ be the root chain of $B_{M^+}$.
Note that $\min C$ has a pair of $D$'s at the endpoints of $e$
and $H$ at every free position,
while $\min C^+$ is obtained from $\min C$ by replacing this $DD$ with $HH$.
Hence,
\[
\min C^+\cover\min C.
\]
Similarly, $\max C^+$ is obtained from $\max C$ by replacing this $DD$ with $VV$,
which implies $\max C\cover\max C^+$.
Therefore, the decomposition has the chain cover property.
\end{proof}

\subsection{Maximum antichains}

A subset $\sA$ of a poset $P$ is an \defn{antichain} if its elements are pairwise incomparable. 
An antichain is called a \defn{maximum antichain} if it has the maximum cardinality among all antichains of $P$.
We use Sperner's theorem together with its equality cases.

\begin{thm}[Sperner's theorem~\cite{Spe28}]
\label{thm:boolean-max-antichain}
Let $n$ be a nonnegative integer.
If $\sA$ is an antichain in the Boolean lattice $\cB_n$, then
\[
|\sA|
\le
\binom{n}{\lfloor n/2\rfloor}.
\]
Moreover, equality holds precisely in the following cases:
\begin{enumerate}[\rm(i)]
    \item If $n$ is even, then $\sA=(\cB_n)_{n/2}$.
    \item If $n$ is odd, then $\sA=(\cB_n)_{(n-1)/2}$ or $\sA=(\cB_n)_{(n+1)/2}$.
\end{enumerate}
\end{thm}

In the following we show that the Delannoy lattice $\cD_n$ satisfies analogous properties.

\begin{thm}\label{thm:unique-antichain-Dn}
Let $n$ be a nonnegative integer.
If $\sA$ is an antichain in the Delannoy lattice $\cD_n$, then $|\sA|\le d(n,\floor{n/2})$.
More precisely,
\begin{enumerate}[\rm(i)]
    \item If $n$ is even, then the unique maximum antichain is $(\cD_n)_{n/2}$.
    \item If $n$ is odd, then the only maximum antichains are $(\cD_n)_{(n-1)/2}$ and $(\cD_n)_{(n+1)/2}$.
\end{enumerate}
\end{thm}

\begin{proof}
For each $0\le i\le\lfloor n/2\rfloor$, let $\sB_i=\{B_M:~|M|=i\}$.
Then
\[
|\mathscr B_i|=\binom{n-i}{i}
\]
and every $B_M\in\mathscr B_i$ is isomorphic to $\cB_{n-2i}$, with bottom rank $i$ in $\cD_n$.

Let $\sA$ be a maximum antichain of $\cD_n$. 
Then for every $B_M\in\sB_i$, the intersection $B_M\cap\sA$ is an antichain in the induced Boolean lattice $B_M$.
Hence, Theorem~\ref{thm:boolean-max-antichain} implies
\begin{equation}\label{eq-BcapA}
    |B_M\cap\sA|
    \le \binom{n-2i}{\floor{(n-2i)/2}}.
\end{equation}
Summing over all components yields
\begin{equation}\label{eq-Aupper}
    |\sA|=\sum_{i=0}^{\floor{n/2}} \sum_{B_M\in\sB_i} |B_M\cap\sA|
    \le\sum_{i=0}^{\floor{n/2}} \binom{n-i}{i} \binom{n-2i}{\floor{(n-2i)/2}}.
\end{equation}
Clearly, $(\cD_n)_{\floor{n/2}}$ is an antichain.
We claim that,
by taking $\sA=(\cD_n)_{\floor{n/2}}$,
the upper bound in \eqref{eq-BcapA} is attained;
therefore, the upper bound in \eqref{eq-Aupper} is also attained.
Indeed, the elements of $B_M$ lying in $(\cD_n)_{\floor{n/2}}$ have rank $\floor{n/2}-i=\floor{(n-2i)/2}$ inside $B_M$,
which implies
\[
\left|B_M\cap(\cD_n)_{\floor{n/2}}\right|=\binom{n-2i}{\floor{(n-2i)/2}}.
\]
Thus, $(\cD_n)_{\floor{n/2}}$ is a maximum antichain.

It remains to prove that the maximum antichains of $\cD_n$ are precisely its middle levels.

(i)
Suppose that $n$ is even. 
Then by Theorem~\ref{thm:boolean-max-antichain}, the intersection $B_M\cap\sA$ is the unique middle level of $B_M$. 
This level has rank $i+(n-2i)/2=n/2$ in $\cD_n$. 
Therefore, $\sA=(\cD_n)_{n/2}$ is the unique antichain.

(ii)
Suppose that $n$ is odd. 
Then by Theorem~\ref{thm:boolean-max-antichain}, the intersection $B_M\cap\sA$ is either the lower middle level or the upper middle level of $B_M$.
These two levels have ranks $(n-1)/2$ and $(n+1)/2$ in $\cD_n$, respectively.
There is a unique Boolean subset $B_{\emptyset}$ with no diagonal steps. 
As the order-reversing involution in Proposition~\ref{prop:self-dual} interchanges the two middle levels of every Boolean subset, we may assume that $\sA$ chooses the lower middle level in $B_{\emptyset}$.

Suppose, for contradiction, that some Boolean subset chooses its upper middle level. 
Let $i\ge1$ be minimal such that a subset $B_M\in\mathscr B_i$ chooses its upper middle level, and choose $p\in B_M\cap\sA$. 
Then $r(p)=(n+1)/2$.
Since $p$ contains $i\ge1$ diagonal steps,
we can replace one occurrence of $D$ in $p$ by $HH$, and denote the resulting path by $p'$.
Thus, $p'$ lies in a Boolean subset $B_{M'}\in\mathscr B_{i-1}$.
By the minimality of $i$, the subset $B_{M'}$ chooses its lower middle level. 
By the cover relation \eqref{eq-cover} we have $p'\cover p$
and $r(p')=(n-1)/2$.
Then $p'$ belongs to the lower middle level, and hence $p'\in\sA$.
This contradicts the assumption that $\sA$ is an antichain,
since $p,p'\in\sA$ and $p'\cover p$.

Therefore, all Boolean subsets make the same choice;
they all choose their lower middle levels under our assumption.
So $\sA=(\cD_n)_{(n-1)/2}$.
Applying the order-reversing involution gives the other possibility $\sA=(\cD_n)_{(n+1)/2}$.
These are the only maximum antichains when $n$ is odd.
\end{proof}

\section{Supersolvability and M\"obius number}
\label{sec:supersolvable}

\subsection{Supersolvability}

Supersolvable lattices were introduced by Stanley~\cite{Sta72}.
Let $L$ be a finite lattice with rank function $r$.
A maximal chain $\Delta$ of $L$ is called an \defn{$M$-chain} if, for every chain $C$ of $L$, the sublattice $K$ generated by $C\cup \Delta$ is \defn{distributive}, that is,
\[
x\meet(y\join z)=(x\meet y)\join(x\meet z), \qquad
\forall x,y,z\in K.
\]
The lattice $L$ is \defn{supersolvable} if it has an $M$-chain.
An element $m\in L$ is \defn{rank-modular} if
\[
r(m)+r(x)=r(m\meet x)+r(m\join x), \qquad
\forall x\in L.
\]

We use the following rank-modular characterization of $M$-chains.

\begin{lem}[{\cite[Theorem 1.4]{FW22}}]
\label{prop:rank-modular-chain}
    Let $L$ be a finite graded lattice. 
    A maximal chain $\Delta$ of $L$ is an $M$-chain if and only if every element of $\Delta$ is rank-modular.
\end{lem}

\begin{thm}
\label{thm:supersolvable}
    The Delannoy lattice $\cD_n$ is supersolvable.
    More precisely, the chain
    \[
    \Delta:\quad H^n=m_0\cover m_1\cover\cdots\cover m_n=V^n,
    \qquad m_i=V^iH^{n-i}
    \]
    is an $M$-chain.
\end{thm}

\begin{proof}
By Lemma~\ref{prop:rank-modular-chain}, it suffices to show that every $m_i\in\Delta$ is rank-modular. 
For any $p\in\cD_n$, we need to verify that
\[
r(m_i)+r(p)=r(m_i\meet p)+r(m_i\join p).
\]

If the path $p$ has a vertex on the line $x+y=i$,
then write
\[
p=ab, \qquad a\in\cD_i, \quad b\in\cD_{n-i}.
\]
By the meet and join formulas established above,
\[
m_i\meet p=aH^{n-i},
\qquad
m_i\join p=V^ib.
\]
Hence,
\[
r(m_i\meet p)+r(m_i\join p)=r(a)+(i+r(b))=i+(r(a)+r(b))=r(m_i)+r(p).
\]

If $p$ crosses the line $x+y=i$ in the interior of a diagonal step,
then write
\[
p=aDb, \qquad a\in\cD_{i-1}, \quad b\in\cD_{n-i-1}.
\]
In this case,
\[
m_i\meet p=aH^{n-i+1},
\qquad
m_i\join p=V^{i+1}b.
\]
Hence,
\[
r(m_i\meet p)+r(m_i\join p)=r(a)+(i+1+r(b))=i+(r(a)+r(b)+1)=r(m_i)+r(p).
\]
Thus every $m_i\in\Delta$ is rank-modular.
By Lemma~\ref{prop:rank-modular-chain}, $\Delta$ is an $M$-chain.
Therefore, $\cD_n$ is supersolvable.
\end{proof}

Bj\"orner \cite{Bjo80} proved that the proper part of a finite supersolvable lattice is shellable,
and shellability implies the Cohen-Macaulay property for its order complex.
Then the following is an immediate consequence.

\begin{cor}
    For $n\ge2$, the order complex of the proper part $\cD_n\backslash\{\hat{0},\hat{1}\}$ is shellable and Cohen--Macaulay.
\end{cor}

\begin{rem}\label{rem-notdis}
Every finite distributive lattice is supersolvable: the distributive identities are inherited by sublattices, so every maximal chain is an $M$-chain. 
However, $\cD_n$ is not distributive for $n\ge2$. 
For example, the interval $[H^n,VVH^{n-2}]$ has three intermediate elements
\[
x=HVH^{n-2},\qquad y=VHH^{n-2},\qquad z=DH^{n-2},
\]
which implies
\[
x\meet(y\join z)=HVH^{n-2}\neq H^n=(x\meet y)\join(x\meet z).
\]
The case $n=2$ is shown in Figure~\ref{fig:D2-M3}.
\end{rem}

\begin{figure}[htbp]
\centering
\tikzset{
    cover/.style={
        draw=blue!70!black,
        line width=1pt,
        shorten <=2pt,
        shorten >=2pt
    },
    pathline/.style={
        draw=black,
        line width=1pt,
        line cap=round,
        line join=round
    },
    dot/.style={
        circle,
        fill=black,
        inner sep=1.5pt
    },
    pathbox/.style={
        rectangle,
        draw=none,
        fill=none,
        inner sep=7pt
    }
}
\tikzset{
pics/HH/.style={code={
    \draw[pathline] (-.7,0)--(0,0)--(.7,0);
    \node[dot] at (-.7,0) {};
    \node[dot] at (0,0) {};
    \node[dot] at (.7,0) {};
}},
pics/HV/.style={code={
    \draw[pathline] (-.35,-.35)--(.35,-.35)--(.35,.35);
    \node[dot] at (-.35,-.35) {};
    \node[dot] at (.35,-.35) {};
    \node[dot] at (.35,.35) {};
}},
pics/VH/.style={code={
    \draw[pathline] (-.35,-.35)--(-.35,.35)--(.35,.35);
    \node[dot] at (-.35,-.35) {};
    \node[dot] at (-.35,.35) {};
    \node[dot] at (.35,.35) {};
}},
pics/D/.style={code={
    \draw[pathline] (-.35,-.35)--(.35,.35);
    \node[dot] at (-.35,-.35) {};
    \node[dot] at (.35,.35) {};
}},
pics/VV/.style={code={
    \draw[pathline] (0,-.7)--(0,0)--(0,.7);
    \node[dot] at (0,-.7) {};
    \node[dot] at (0,0) {};
    \node[dot] at (0,.7) {};
}}
}

\newcommand{\dpathnode}[3]{%
    \node[pathbox] (#1) at #2 {%
        \begin{tikzpicture}[baseline=-0.5ex]
            \pic {#3};
        \end{tikzpicture}%
    };%
}

\scalebox{0.85}{
\begin{tikzpicture}

\dpathnode{B0}{(0,0)}{HH}

\dpathnode{B1L}{(-2.2,2.2)}{HV}
\dpathnode{B1M}{(0,2.2)}{D}
\dpathnode{B1R}{(2.2,2.2)}{VH}

\dpathnode{B2}{(0,4.7)}{VV}

\draw[cover] (B0.north)--(B1L.south);
\draw[cover] (B0.north)--(B1M.south);
\draw[cover] (B0.north)--(B1R.south);

\draw[cover] (B1L.north)--(B2.south);
\draw[cover] (B1M.north)--(B2.south);
\draw[cover] (B1R.north)--(B2.south);

\draw[red!70!black, line width=1.2pt, rounded corners]
    (B0.south west) rectangle (B0.north east);

\draw[red!70!black, line width=1.2pt, rounded corners]
    (B1L.south west) rectangle (B1L.north east);

\draw[red!70!black, line width=1.2pt, rounded corners]
    (B1M.south west) rectangle (B1M.north east);

\draw[red!70!black, line width=1.2pt, rounded corners]
    (B1R.south west) rectangle (B1R.north east);


\node[red!70!black] at (-4.2,1.2) {$x=x\meet(y\join z)=HV$};
\node[red!70!black] at (-0.66,1.2) {$z=D$};
\node[red!70!black] at (2.7,1.2) {$y=VH$};
\node[red!70!black] at (0,-0.7) {$(x\meet y)\join(x\meet z)=HH$};

\end{tikzpicture}
}
\caption{A counterexample to distributivity in $\cD_2$.}
\label{fig:D2-M3}
\end{figure}
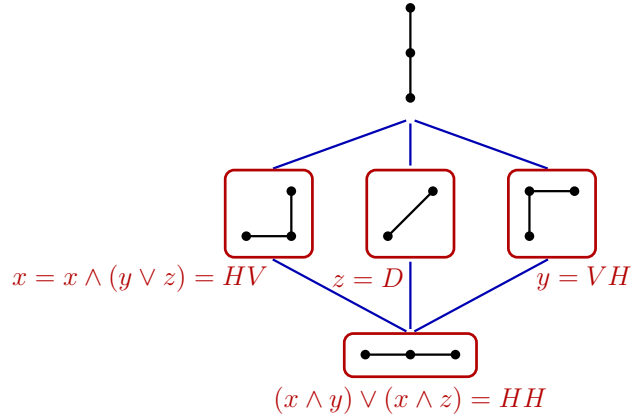

\begin{rem}
Every distributive lattice is modular, but the converse need not hold.
For $n\ge3$, the lattice $\cD_n$ is not modular.
For example, let
\[
m=HDH^{n-3},\qquad x=DHH^{n-3}.
\]
Then
\[
m\meet x=H^n,\qquad m\join x=VVVH^{n-3},
\]
and therefore
\[
r(m)+r(x)=2
\neq
3=r(m\meet x)+r(m\join x).
\]
Thus $m$ is not rank-modular.
Figure~\ref{fig:D3-nonmodular} shows the case $n=3$.
\end{rem}

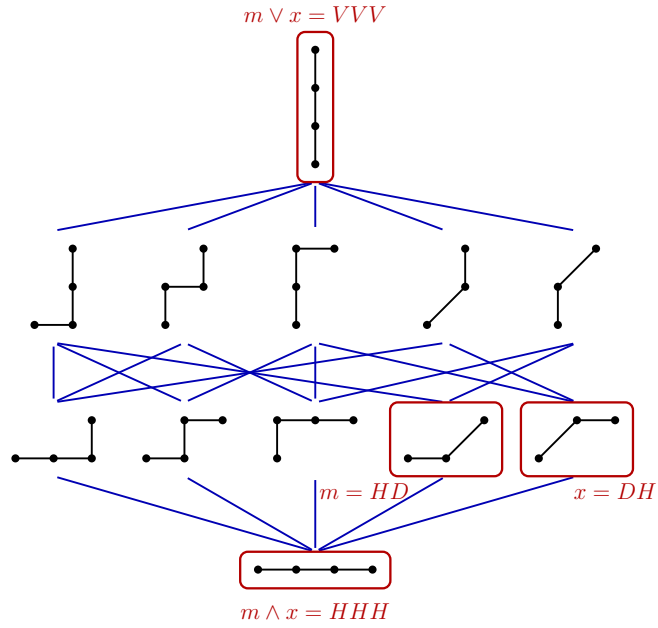
\begin{figure}[htbp]
\centering
\tikzset{
    cover/.style={
        draw=blue!70!black,
        line width=1pt,
        shorten <=2pt,
        shorten >=2pt
    },
    pathline/.style={
        draw=black,
        line width=1pt,
        line cap=round,
        line join=round
    },
    dot/.style={
        circle,
        fill=black,
        inner sep=1.5pt
    },
    pathbox/.style={
        rectangle,
        draw=none,
        fill=none,
        inner sep=7pt
    }
}

\tikzset{
pics/HHH/.style={code={
    \draw[pathline] (-1.05,0)--(-.35,0)--(.35,0)--(1.05,0);
    \foreach \x in {-1.05,-.35,.35,1.05}
        \node[dot] at (\x,0) {};
}},
pics/HHV/.style={code={
    \draw[pathline] (-.7,-.35)--(0,-.35)--(.7,-.35)--(.7,.35);
    \node[dot] at (-.7,-.35) {};
    \node[dot] at (0,-.35) {};
    \node[dot] at (.7,-.35) {};
    \node[dot] at (.7,.35) {};
}},
pics/HVH/.style={code={
    \draw[pathline] (-.7,-.35)--(0,-.35)--(0,.35)--(.7,.35);
    \node[dot] at (-.7,-.35) {};
    \node[dot] at (0,-.35) {};
    \node[dot] at (0,.35) {};
    \node[dot] at (.7,.35) {};
}},
pics/VHH/.style={code={
    \draw[pathline] (-.7,-.35)--(-.7,.35)--(0,.35)--(.7,.35);
    \node[dot] at (-.7,-.35) {};
    \node[dot] at (-.7,.35) {};
    \node[dot] at (0,.35) {};
    \node[dot] at (.7,.35) {};
}},
pics/HD/.style={code={
    \draw[pathline] (-.7,-.35)--(0,-.35)--(.7,.35);
    \node[dot] at (-.7,-.35) {};
    \node[dot] at (0,-.35) {};
    \node[dot] at (.7,.35) {};
}},
pics/DH/.style={code={
    \draw[pathline] (-.7,-.35)--(0,.35)--(.7,.35);
    \node[dot] at (-.7,-.35) {};
    \node[dot] at (0,.35) {};
    \node[dot] at (.7,.35) {};
}},
pics/HVV/.style={code={
    \draw[pathline] (-.35,-.7)--(.35,-.7)--(.35,0)--(.35,.7);
    \node[dot] at (-.35,-.7) {};
    \node[dot] at (.35,-.7) {};
    \node[dot] at (.35,0) {};
    \node[dot] at (.35,.7) {};
}},
pics/VHV/.style={code={
    \draw[pathline] (-.35,-.7)--(-.35,0)--(.35,0)--(.35,.7);
    \node[dot] at (-.35,-.7) {};
    \node[dot] at (-.35,0) {};
    \node[dot] at (.35,0) {};
    \node[dot] at (.35,.7) {};
}},
pics/VVH/.style={code={
    \draw[pathline] (-.35,-.7)--(-.35,0)--(-.35,.7)--(.35,.7);
    \node[dot] at (-.35,-.7) {};
    \node[dot] at (-.35,0) {};
    \node[dot] at (-.35,.7) {};
    \node[dot] at (.35,.7) {};
}},
pics/DV/.style={code={
    \draw[pathline] (-.35,-.7)--(.35,0)--(.35,.7);
    \node[dot] at (-.35,-.7) {};
    \node[dot] at (.35,0) {};
    \node[dot] at (.35,.7) {};
}},
pics/VD/.style={code={
    \draw[pathline] (-.35,-.7)--(-.35,0)--(.35,.7);
    \node[dot] at (-.35,-.7) {};
    \node[dot] at (-.35,0) {};
    \node[dot] at (.35,.7) {};
}},
pics/VVV/.style={code={
    \draw[pathline] (0,-1.05)--(0,-.35)--(0,.35)--(0,1.05);
    \foreach \y in {-1.05,-.35,.35,1.05}
        \node[dot] at (0,\y) {};
}}
}

\newcommand{\dpathnode}[3]{%
    \node[pathbox] (#1) at #2 {%
        \begin{tikzpicture}[baseline=-0.5ex]
            \pic {#3};
        \end{tikzpicture}%
    };%
}

\scalebox{0.72}{
\begin{tikzpicture}

\dpathnode{C0}{(0,0)}{HHH}

\dpathnode{C1a}{(-4.8,2.4)}{HHV}
\dpathnode{C1b}{(-2.4,2.4)}{HVH}
\dpathnode{C1c}{(0,2.4)}{VHH}
\dpathnode{C1d}{(2.4,2.4)}{HD}
\dpathnode{C1e}{(4.8,2.4)}{DH}

\dpathnode{C2a}{(-4.8,5.2)}{HVV}
\dpathnode{C2b}{(-2.4,5.2)}{VHV}
\dpathnode{C2c}{(0,5.2)}{VVH}
\dpathnode{C2d}{(2.4,5.2)}{DV}
\dpathnode{C2e}{(4.8,5.2)}{VD}

\dpathnode{C3}{(0,8.5)}{VVV}

\draw[cover] (C0.north)--(C1a.south);
\draw[cover] (C0.north)--(C1b.south);
\draw[cover] (C0.north)--(C1c.south);
\draw[cover] (C0.north)--(C1d.south);
\draw[cover] (C0.north)--(C1e.south);

\draw[cover] (C1a.north)--(C2a.south);
\draw[cover] (C1a.north)--(C2b.south);
\draw[cover] (C1a.north)--(C2d.south);

\draw[cover] (C1b.north)--(C2a.south);
\draw[cover] (C1b.north)--(C2c.south);

\draw[cover] (C1c.north)--(C2b.south);
\draw[cover] (C1c.north)--(C2c.south);
\draw[cover] (C1c.north)--(C2e.south);

\draw[cover] (C1d.north)--(C2a.south);
\draw[cover] (C1d.north)--(C2e.south);

\draw[cover] (C1e.north)--(C2c.south);
\draw[cover] (C1e.north)--(C2d.south);

\draw[cover] (C2a.north)--(C3.south);
\draw[cover] (C2b.north)--(C3.south);
\draw[cover] (C2c.north)--(C3.south);
\draw[cover] (C2d.north)--(C3.south);
\draw[cover] (C2e.north)--(C3.south);

\draw[red!70!black, line width=1.2pt, rounded corners]
    (C0.south west) rectangle (C0.north east);

\draw[red!70!black, line width=1.2pt, rounded corners]
    (C1d.south west) rectangle (C1d.north east);

\draw[red!70!black, line width=1.2pt, rounded corners]
    (C1e.south west) rectangle (C1e.north east);

\draw[red!70!black, line width=1.2pt, rounded corners]
    (C3.south west) rectangle (C3.north east);

\node[red!70!black] at (0,-0.75) {$m\meet x=HHH$};
\node[red!70!black] at (0.9,1.45) {$m=HD$};
\node[red!70!black] at (5.5,1.45) {$x=DH$};
\node[red!70!black] at (0,10.2) {$m\join x=VVV$};

\end{tikzpicture}
}

\caption{A counterexample to modularity in $\cD_3$.}
\label{fig:D3-nonmodular}
\end{figure}

\subsection{The M\"obius number}
\label{subsec:mobius}


The \defn{M\"obius function} $\mu$ of a finite poset $P$ is defined recursively by $\mu(s,s)=1$ and
\[
\mu(s,t)=-\sum_{s\le u<t}\mu(s,u), \quad s<t.
\]
If $P$ has a minimum element $\hat 0$ and a maximum element $\hat 1$, then
\[
\mu(P):=\mu_P(\hat 0,\hat 1)
\]
is called the \defn{M\"obius number} of $P$.

An \defn{edge-labeling} of $P$ is a map $\lam$ that assigns an integer to each cover relation.
The label sequence of a saturated chain $C: x_0\cover x_1\cover\cdots\cover x_k$ is
\[
\lambda(C)=\bigl(\lambda(x_0,x_1),\lambda(x_1,x_2),\ldots,\lambda(x_{k-1},x_k)\bigr).
\]
We say that the chain $C$ is \defn{increasing} (resp. \defn{decreasing}) if its label sequence is weakly increasing (resp. decreasing). 
An edge-labeling is called an \defn{$R$-labeling} if every interval has a unique increasing saturated chain.

\begin{lem}[{\cite[Section 3.14]{Sta12}}]
\label{thm:R-labeling-mobius}
Let $P$ be a supersolvable lattice of rank $n$,
and let $\hat 0=m_0\cover m_1\cover\cdots\cover m_n=\hat 1$ be an $M$-chain.
For $s\cover t$, define
\begin{equation}\label{eq-label}
    \lambda(s,t)=\min\{i:s\join m_i=t\join m_i\}.
\end{equation}
Then $\lam$ is an $R$-labeling with label set $[n]$,
and the labels along every maximal chain form a permutation of $[n]$.
Moreover,
\[
\mu(P)=(-1)^n\cdot\#\{\text{strictly decreasing maximal chains of }P\}.
\]
\end{lem}

\begin{thm}
\label{thm:mobius-Dn}
For every $n\ge0$, the M\"obius number of $\cD_n$ is
\[
\mu(\cD_n)
=(-1)^nF_{n+1},
\]
where $(F_n)_{n\ge0}$ is the Fibonacci sequence with $F_0=0$ and $F_1=1$.
\end{thm}

\begin{proof}
We first apply Lemma \ref{thm:R-labeling-mobius} to the $M$-chain
\[
\Delta:\quad H^n=m_0\cover m_1\cover\cdots\cover m_n=V^n,
\qquad m_i=V^iH^{n-i}
\]
in Theorem \ref{thm:supersolvable} and interpretate each strictly decreasing maximal chain of $\cD_n$ as a composition of $n$ whose parts are $1$ or $2$.

Every strictly decreasing maximal chain has exactly $n$ cover relations,
so its label sequence must be $(n,n-1,\ldots,1)$.
We read such a chain from $H^n$ to $V^n$.
Suppose that the last $k$ steps have been changed to vertical steps
and the current path is $H^{n-k}V^k$ for $0\le k\le n-1$.
Then the next required lable is $n-k$.
It is not hard to verify that there are exactly two possibilities for the next cover relation.
One may replace the rightmost remaining $H$ by $V$, i.e.,
\[
H^{n-k}V^k\cover H^{n-k-1}V^{k+1}.
\]
In this case, we record a part $1$.
Alternatively, one may replace the two rightmost remaining $H$'s by $D$, i.e.,
\[
H^{n-k}V^k\cover H^{n-k-2}DV^k.
\]
The next required lable is $n-k-1$,
which gives rise to the cover relation
\[
H^{n-k-2}DV^k\cover H^{n-k-2}V^{k+2}.
\]
In this case, we record a part $2$.
Repeating this procedure leads to a composition of $n$ with parts $1$ and $2$.
Conversely, given such a composition,
scan its parts from left to right.
For a part $1$, change the rightmost $H$ to $V$,
while for a part $2$, perform $HH\cover D\cover VV$ on the two rightmost $H$'s.
Then the above argument implies that the reaulting lables are $n,n-1,\dots,1$.
Thus the resulting maximal chain is strictly decreasing.

Next, we consider the enumeration.
Let $a_n$ be the number of $\{1,2\}$-compositions of $n$.
Removing the last part of a composition yields $a_n=a_{n-1}+a_{n-2}$ with $a_0=a_1=1$.
Thus $a_n=F_{n+1}$ as desired.
Therefore,
\[
\mu(\cD_n)=(-1)^n a_n=(-1)^n F_{n+1}
\]
by Lemma \ref{thm:R-labeling-mobius}.
\end{proof}

\begin{exa}
For $n=4$, the composition $(1,2,1)$ corresponds to the strictly decreasing maximal chain
\[
HHHH \xto{4} HHHV \xto{3} HDV \xto{2} HVVV \xto{1} VVVV,
\]
where $\xto{*}$ indicates the cover relations with label $*$ above it.
\end{exa}

\section{Characteristic and zeta polynomials}

\subsection{Characteristic polynomials}
Let $P$ be a finite graded poset of rank $n$ with a minimum element $\hat 0$.
The \defn{characteristic polynomial} of $P$ is defined as
\[
\chi_P(\lambda)=\sum_{x\in P}\mu(\hat 0,x)\lambda^{n-r(x)}.
\]

We use the following recurrence due to Greene \cite{Gre88}.

\begin{prop}[\cite{Gre88}]
\label{prop:Greene}
Let $L$ be a finite graded lattice of rank $n$,
and let $b$ be a coatom ($b\cover\hat{1}$) of $L$.
Then
\[
\chi_L(\lambda)=\lambda\,\chi_{[\hat 0,b]}(\lambda)
-
\sum_{\substack{x>\hat 0\\ x\meet b=\hat 0}}
\chi_{[x,\hat 1]}(\lambda).
\]
\end{prop}

\begin{thm}
    Let $\chi_n(\lam)$ be the characteristic polynomial of $\cD_n$.
    Then
    \begin{equation}\label{eq-char-formula}
        \chi_n(\lambda)=\sum_{k=0}^{\floor{n/2}} (-1)^k \binom{n-k}{k} (\lambda-1)^{n-k}.
    \end{equation}
\end{thm}

\begin{proof}
We first prove that $\chi_n(\lam)$ satisfies the recurrence
\begin{equation}\label{eq-char-rec}
    \chi_n(\lambda)=(\lambda-1)\bigl(\chi_{n-1}(\lambda)-\chi_{n-2}(\lambda)\bigr),
    \qquad n\ge 2,
\end{equation}
with $\chi_0(\lambda)=1$ and $\chi_1(\lambda)=\lambda-1$.

Let $b=V^{n-1}H$.
This is a coatom of $\cD_n$ since $b\cover\hat{1}=V^n$.
Then an element lies in $[\hat 0,b]$ if and only if it ends in $H$.
Deleting this final $H$ gives an isomorphism $[\hat 0,b]\cong \cD_{n-1}$.
Hence, $\chi_{[\hat 0,b]}(\lambda)=\chi_{n-1}(\lambda)$.
Theorem~\ref{thm:supersolvable} implies that $b$ is rank-modular.
So for each $x>\hat 0$ with $x\meet b=\hat 0$ we have
\[
r(x)+r(b)=r(x\meet b)+r(x\join b)=r(x\join b).
\]
Since $b$ is a coatom, $r(b)=n-1$ and $r(x\join b)\le n$.
Then
\[
r(x)+n-1=r(x\join b)\le n.
\]
Hence, $r(x)=1$, i.e., $x$ must be an atom.
Since $x>\hat 0$ and $x\meet b=\hat 0$, we have $x\notin[\hat 0,b]$.
It follows that the path $x$ does not end in $H$.
Thus $x=H^{n-1}V$ or $x=H^{n-2}D$.

If $x=H^{n-1}V$, then an element lies in $[x,\hat 1]$ if and only if it ends in $V$.
Deleting this final $V$ gives an isomorphism $[x_1,\hat 1]\cong\cD_{n-1}$.
Thus, $\chi_{[x_1,\hat 1]}(\lambda)=\chi_{n-1}(\lambda)$.
If $x=H^{n-2}D$, then elements in $[x,\hat 1]$ must end in $D$ or have $VV$ as the last two steps.
So $[x,\hat 1]\cong\cD_{n-2}\times\cB_1$,
where $\cB_1\cong\{D<VV\}$.
It is known that the characteristic polynomial of the Cartesian product is the product of those of the two posets.
Note that $\chi_{\cB_1}(\lam)=\lam-1$.
Hence, $\chi_{[x,\hat 1]}(\lambda)=(\lambda-1)\chi_{n-2}(\lambda)$.
Therefore, applying Proposition~\ref{prop:Greene} yields
\[
\begin{aligned}
    \chi_n(\lambda)
    =
    \lambda\chi_{n-1}(\lambda)
    -\chi_{n-1}(\lambda)
    -(\lambda-1)\chi_{n-2}(\lambda)
    =
    (\lambda-1)
    \bigl(\chi_{n-1}(\lambda)-\chi_{n-2}(\lambda)\bigr).
\end{aligned}
\]
The initial conditions follow from $\cD_0\cong\cB_0$ and $\cD_1\cong\cB_1$.

Next, by taking generating function of the recurrence \eqref{eq-char-rec}, we obtain
$$
\sum_{n\ge0}\chi_n(\lambda)x^n=\frac{1}{1-(\lambda-1)x+(\lambda-1)x^2}.
$$
Then
\[
\sum_{n\ge0}\chi_n(\lambda)x^n
=\frac{1}{1-(\lambda-1)x(1-x)}
=\sum_{m\ge0} (\lambda-1)^m x^m(1-x)^m.
\]
Extracting the coefficients of $x^n$ on both sides yields \eqref{eq-char-formula}.
\end{proof}

Suppose that $f(x)$ and $g(x)$ are two polynomials that have only real zeros.
Let $(r_i)_{1\le i\le n}$ and $(s_i)_{1\le i\le m}$ be all roots of $f(x)$ and $g(x)$ in weakly decreasing order, respectively.
Say that $g(x)$ \defn{interlaces} $f(x)$ if $m\le n\le m+1$ and $r_{i+1}\le s_i\le r_i$ for all $i$.

\begin{thm}
\label{thm:char-real-rooted}
The characteristic polynomial $\chi_n(\lambda)$ has only real zeros, all lying in $[1,\infty)$. 
More precisely,
\[
\chi_n(\lambda)
=
(\lambda-1)^{\lceil n/2\rceil}
\prod_{j=1}^{\lfloor n/2\rfloor}
\left(
\lambda-1-4\cos^2\frac{j\pi}{n+1}
\right).
\]
In particular, $\chi_{n-1}(\lambda)$ interlaces $\chi_n(\lambda)$ for all $n\ge1$.
\end{thm}

\begin{proof}
Set $y=\lambda-1$ and $p_n(y)=\chi_n(y+1)$.
By Corollary~\ref{cor:char-explicit},
\[
p_n(y)=\sum_{k=0}^{\floor{n/2}} (-1)^k \binom{n-k}{k} y^{n-k}.
\]
Let $U_n(x)$ denote the Chebyshev polynomial of the second kind. 
Since
\[
U_n(x/2)=\sum_{k=0}^{\floor{n/2}} (-1)^k \binom{n-k}{k} x^{n-2k},
\]
we have $p_n(x^2)=x^nU_n(x/2).$
The zeros of $U_n(x/2)$ are $2\cos\frac{j\pi}{n+1}$ for $1\le j\le n$.
Pairing the conjugate zeros and then setting $y=x^2$ gives
\[
p_n(y)=y^{\lceil n/2\rceil} \prod_{j=1}^{\floor{n/2}} \left(y-4\cos^2\frac{j\pi}{n+1}\right).
\]
Substituting $y=\lambda-1$ implies the factorization and shows that all zeros lie in $[1,\infty)$.

It is known that the Chebyshev polynomials $U_{n-1}(x)$ interlaces $U_n(x)$.
After pairing conjugate zeros, applying $x\mapsto x^2$, and adjoining the multiplicities at $y=0$,
these interlacing inequalities become interlacing for $p_{n-1}(y)$ and $p_n(y)$.
Therefore, $\chi_{n-1}(\lambda)$ interlaces $\chi_n(\lambda)$
since translation by one preserves the interlacing inequalities.
\end{proof}

\subsection{Zeta polynomials}

Let $P$ be a finite poset.
For an integer $s\ge2$, define $Z(P,s)$ to be the number of multichains $t_1\le\cdots\le t_{s-1}$ in $P$.
We call $Z(P,s)$ the \defn{zeta polynomial} of $P$, regarded as a function of $s$.
Write $Z_n(s)=Z(\cD_n,s)$.

\begin{thm}\label{thm-zeta-explicit}
    Let $Z_n(s)$ be the zeta polynomial of $\cD_n$. then
    \begin{equation}\label{eq-zeta-formula}
        Z_n(s)=\sum_{k=0}^{\floor{n/2}} \binom{n-k}{k} s^{n-2k} \binom{s}{2}^k.
    \end{equation}
\end{thm}

\begin{proof}
For $s\ge2$, we first prove that $Z_n(s)$ satisfies the recurrence
\begin{equation}\label{eq-zeta-rec}
    Z_n(s)=sZ_{n-1}(s)+\binom{s}{2}Z_{n-2}(s), \qquad n\ge2,
\end{equation}
with $Z_0(s)=1$ and $Z_1(s)=s$.
Let
\[
\boldsymbol{t}=\{t_1\le t_2\le\cdots\le t_{s-1}\}
\]
be a multichain in $\cD_n$.
We distinguish two cases according to the terminal steps of the Delannoy paths in $\boldsymbol{t}$.

Suppose first that none of $t_1,\dots,t_{s-1}$ ends in $D$. 
Then every path in $\boldsymbol{t}$ ends in either $H$ or $V$. 
Deleting the final step from every path gives a multichain in $\cD_{n-1}$.
The deleted steps form a weakly increasing sequence of length $s-1$ in the two-element chain $H<V$.
Such a sequence is determined by the number of initial $H$'s,
which may be any integer from $0$ to $s-1$.
So there are $s$ possible sequences.
Hence, this case contributes $sZ_{n-1}(s)$.

Suppose now that some path in $\boldsymbol{t}$ ends in $D$.
Then the final step(s) of $t_1,\dots,t_{s-1}$ are $HH$, $D$, or $VV$.
Deleting these steps from every path gives a multichain in $\cD_{n-2}$.
The deleted block form a weakly increasing sequence of length $s-1$ in the chain $HH<D<VV$
that contains at least one $D$.
Thus the sequence has the form
\[
\underbrace{HH,\ldots,HH}_{i}<\underbrace{D,\ldots,D}_{j}<\underbrace{VV,\ldots,VV}_{k},
\]
where
\[
i,k\ge0,\quad j\ge1,\quad i+j+k=s-1.
\]
There are $\binom{(s-2)+3-1}{s-2}=\binom{s}{2}$ such triples $(i,j,k)$.
Thus, this case contributes $\binom{s}{2}Z_{n-2}(s)$.

The two cases are disjoint and exhaustive, so the recurrence holds for every positive integer $s\ge2$,
and hence identically.
The initial values are immediate.

Then by taking generating function of the recurrence \eqref{eq-zeta-rec}, we derive
$$
\sum_{n\ge0} Z_n(s)x^n=\frac{1}{1-sx-\binom{s}{2}x^2}.
$$
Hence,
\[
\sum_{n\ge0} Z_n(s)x^n=\sum_{m\ge0} x^m \left(s+\binom{s}{2}x\right)^m.
\]
Extracting the coefficients of $x^n$ on both sides gives \eqref{eq-zeta-formula}.
\end{proof}

\begin{rem}
Note that $Z(P,2)=|P|$ for every finite poset $P$.
Setting $s=2$ in Theorem \ref{thm-zeta-explicit} gives
\[
|\cD_n|=\sum_{k=0}^{\floor{n/2}} \binom{n-k}{k} 2^{n-2k}.
\]
So $|\cD_n|=P_{n+1}$, where $P_n$ is the $n$th Pell number with $P_1=1$ and $P_2=2$.
It is known \cite[Proposition 3.12.1]{Sta12} that $Z(P,-1)=\mu(P)$.
Setting $s=-1$ in Theorem \ref{thm-zeta-explicit} gives
\[
\mu(\cD_n)=\sum_{k=0}^{\floor{n/2}} \binom{n-k}{k} (-1)^{n-2k}
=(-1)^n \sum_{k=0}^{\floor{n/2}} \binom{n-k}{k}
=(-1)^nF_{n+1},
\]
where $F_n$ is the $n$th Fibonacci number.
This agrees with Theorem~\ref{thm:mobius-Dn}.
\end{rem}

\section{Further questions}

Let $P$ be a finite graded poset of rank $r(P)$.
Recall that $P_k$ denotes its rank $k$ layer.
Let
\[
\Gamma(X)=\{y\in P:x\cover y\text{ for some }x\in X\}.
\]
We say that $P$ is \defn{normal},
or $P$ has the \defn{normalized matching property},
if for every $X\subseteq P_k$,
\begin{equation}\label{eq-normal}
    \frac{|\Gamma(X)|}{|P_{k+1}|}\ge\frac{|X|}{|P_k|},
    \qquad 0\le k< r(P).
\end{equation}

Anderson and Griggs independently proved that every rank-symmetric, rank-unimodal normal poset admits a symmetric chain decomposition~\cite{And02,Grig77}.
Proposition~\ref{prop-SCD} shows that the Delannoy lattice $\cD_n$ admits a symmetric chain decomposition and hence is rank-symmetric and rank-unimodal.
It is therefore natural to ask whether $\cD_n$ is also normal.

\begin{prob}
Is the Delannoy lattice $\cD_n$ normal for every $n\ge0$?
\end{prob}

The answer is affirmative for $n\le5$.

\begin{thm}\label{thm:Dn-normal-small}
The Delannoy lattice $\cD_n$ is normal for $0\le n\le5$.
\end{thm}

The proof of Theorem \ref{thm:Dn-normal-small} proceeds by establishing the normalized matching inequality \eqref{eq-normal} between the rank 1 and rank 2 layers of $\cD_n$ for all $n$, and then applying the self-duality of $\cD_n$.
For the sake of brevity, the details of the proof are deferred to the Appendix.

\appendix
\section{Proof of Theorem \ref{thm:Dn-normal-small}}

\begin{lem}\label{lem:rank12-normal}
    For every $n\ge2$, the normalized matching inequality \eqref{eq-normal} holds between $(\cD_n)_1$ and $(\cD_n)_2$.
\end{lem}

\begin{proof}
Fix $n$, denote by $P_1$ and $P_2$ the rank $1$ and rank $2$ layers of $\cD_n$.
Elements of $P_1$ consists of Delannoy paths obtained from $H^n$ by either replacing one $H$ by one $V$
(with $n$ choices),
or replacing two consecutive $H$'s by one $D$
(with $n-1$ choices).
So
\[
a:=|P_1|=2n-1.
\]
Similar arguments gives
\[
b:=|P_2|=\binom n2+(n-1)(n-2)+\binom{n-2}{2}=2n^2-6n+5,
\]
where the three terms count Delannoy paths with exactly two $V$'s, one $V$ and one $D$, and two $D$'s, respectively.
It suffices to show that
\begin{equation}\label{eq-ab}
    |\Ga(X)|\ge\frac{b}{a} |X|.
\end{equation}
Note that
\begin{equation}\label{eq-ineq}
    \frac{2n-5}{2}<\frac{b}{a}=\frac{2n-5}{2}+\frac{5}{2(2n-1)}<\frac{2n-3}{2}.
\end{equation}

\textbf{Claim 1.}
For each $x\in P_1$, there are at least $n-1$ elements in $P_2$ that cover $x$.

\textit{Proof.}
Every $x\in P_1$ consists of
either one $V$ and $n-1$ occurrences of $H$,
or one $D$ and $n-2$ occurrences of $H$.
Each $H$ can be changed to $V$, and each $D$ can be changed to $VV$.
So there are at least $n-1$ choices to obtain an element that cover $x$.
\qed

\textbf{Claim 2.}
For distinct $x,x'\in P_1$, there are at most one element in $P_2$ that cover both $x$ and $x'$.

\textit{Proof.}
If $y\in P_2$ covers distinct $x,x'\in P_1$, then $x\join x'\le y$.
Since $x\ne x'$, we have $r(x\join x')\ge2$. 
As $r(y)=2$, it follows that $x\join x'=y$.
An element covering both $x$ and $x'$ may appear in any layer of rank $\ge2$,
but such an element that appears in $P_2$ is uniquely determined by $x$ and $x'$.
\qed

\textbf{Claim 3.}
Every element in $P_2$ covers at least two distinct elements in $P_1$.

\textit{Proof.}
The non-$H$ steps of $y\in P_2$ have exactly three possibilities:
two independent $V$ steps,
two independent $D$ steps,
and one $V$ and one $D$.
Then an $x\in P_1$ covered by $y$ is obtained from $y$ by
either replacing one $V$ by $H$
or replacing one $D$ by $HH$.
So in each of the possibilities, $x$ can be obtained in two different ways.
\qed

Let $X\subseteq P_1$.
There is nothing to prove if $X=\emptyset$ or $X=P_1$.
Now suppose that $0<|X|<a$.
If $|X|=1$, then Claim 1 implies
\[
|\Gamma(X)|\ge n-1>\frac{a}{b}.
\]
Thus \eqref{eq-ab} is true whenever $n\ge2$.
If $|X|=2$, then combining Claim 1 and Claim 2 yields
\[
|\Gamma(X)|\ge2(n-1)-1=2n-3>\frac{2b}{a}=\frac ba |X|.
\]
Suppose $|X|\ge3$.
Set $T=P_1\setminus X$, and $|T|=t$.
Let $y\in P_2\setminus\Gamma(X)$.
Then the elements covered by $y$ all belong to $T$.
By Claim 3 we can choose two distinct such elements $u_y$ and $v_y$.
Then the map $y\longmapsto\{u_y,v_y\}$ is injective by Claim 2. 
Therefore,
\[
|P_2\setminus\Gamma(X)|\le\binom{t}{2}.
\]
Note that
\[
t=a-|X|\le a-3=2n-4.
\]
Combining this with \eqref{eq-ineq} implies
\[
\binom{t}{2}
=\frac{t(t-1)}{2}
\le\frac{2n-5}{2}\,t
\le\frac{b}{a}\,t.
\]
Hence,
\[
|\Gamma(X)|=b-|P_2\setminus\Gamma(X)|\ge b-\frac{b}{a}t=\frac{b}{a}(a-t)=\frac{b}{a}|X|.
\]
This proves the desired inequality \eqref{eq-ab}.
\end{proof}

\begin{proof}[Proof of Theorem \ref{thm:Dn-normal-small}]
We first note that normalized matching inequalities are preserved
under duality.
Suppose that the inequality holds between $P_k$ and $P_{k+1}$,
and let $Y\subseteq P_{k+1}$.
Denote
\[
\Gamma^-(Y)=\{x\in P:x\cover y\text{ for some }y\in Y\}.
\]
Taking $X=P_k\setminus\Gamma^-(Y)$ gives
$\Gamma(X)\subseteq P_{k+1}\setminus Y$, and hence
\[
1-\frac{|\Gamma^-(Y)|}{|P_k|}
=\frac{|X|}{|P_k|}
\le\frac{|\Gamma(X)|}{|P_{k+1}|}
\le1-\frac{|Y|}{|P_{k+1}|}.
\]
This is the normalized matching inequality in the reversed direction.

The inequalities involving rank $0$ or rank $n$ are trivial.
Lemma~\ref{lem:rank12-normal} gives the inequality between
ranks $1$ and $2$ for $n\ge2$.
By the self-duality in Proposition~\ref{prop:self-dual},
it also gives the inequality between ranks $n-2$ and $n-1$. This proves the result for $n\le4$.

For $n=5$, it remains only to consider ranks $2$ and $3$.
The symmetric chain decomposition in Proposition~\ref{prop-SCD} induces a perfect matching from $(\cD_5)_2$ to $(\cD_5)_3$.
Thus $|\Gamma(X)|\ge|X|$ for every $X\subseteq(\cD_5)_2$.
Since the two ranks have the same size, this proves the normalized matching inequality.
\end{proof}

\section*{Acknowledgements}
We thank Yi Wang for helpful discussions and suggestions.
This work was supported in part by the National Natural Science Foundation of China (Nos. 12671384, 12571348).

\end{document}